\documentclass{article}
\usepackage{amsmath,amsthm,amssymb,amsfonts}
\usepackage{verbatim}
\usepackage{graphicx}
\usepackage{stmaryrd} 
\usepackage{hyperref}

\usepackage[colorinlistoftodos]{todonotes}

\theoremstyle{plain}
\newtheorem{thm}{Theorem}
\newtheorem{lem}[thm]{Lemma}
\newtheorem{prop}[thm]{Proposition}
\newtheorem{cor}[thm]{Corollary}

\newtheorem{remark}[thm]{Remark}

\theoremstyle{definition}
\newtheorem{definition}[thm]{Definition}

\numberwithin{thm}{section}

\newcommand{\adj}{\leftrightarrow}
\newcommand{\adjeq}{\leftrightarroweq}

\DeclareMathOperator{\id}{id}
\DeclareMathOperator{\Fix}{Fix}
\def\Z{{\mathbb Z}}
\def\N{{\mathbb N}}

\title{Freezing and Cold Sets for Digital Cones and Suspensions}
\author{Laurence Boxer
\thanks{
    Department of Computer and Information Sciences,
    Niagara University,
    Niagara University, NY 14109, USA;
    and Department of Computer Science and Engineering,
    State University of New York at Buffalo.
    email: boxer@niagara.edu
}
}

\date{ }
\begin{document}
\maketitle{}

\begin{abstract}
Cone and suspension
constructions have been
introduced in digital
topology, modeled on those
of classical topology.
For digital cones and 
suspensions, and for some
related digital images, we
find freezing and cold sets.


\end{abstract}

\maketitle

\section{Introduction}
Although in some ways
digitally continuous functions on
digital images are similar
to classical topology's
continuous functions on
subsets of Euclidean spaces,
the locally finite nature of
digital images often
restricts digitally
continuous functions in
ways that classical continuous
functions are not restricted.
The fixed point theory of
digital images has yielded
many such restrictions, including
freezing sets and cold sets.

For a finite digital image $(X,\kappa)$, it is often
possible to compute a freezing or cold set~$A$ such
that $\#A$ is much less than~$\#X$ and
the behavior of a continuous self-map~$f$ on~$X$
is greatly influenced by the behavior of~$f|_A$.
We study freezing sets and cold sets
for digital cones, suspensions, and some 
cone-like and suspension-like digital images.

\section{Preliminaries}
We use $\N$ for the set of natural numbers,
$\N^* = \N \cup \{0\}$,
$\Z$ for the set of integers, and
$\#X$ for the number of distinct members of $X$.

We typically denote a (binary) digital image
as $(X,\kappa)$, where $X$ is a nonempty subset
of $\Z^n$ for some
$n \in \N$ and $\kappa$ represents an adjacency
relation of pairs of distinct points in $X$. Thus,
$(X,\kappa)$ is an undirected graph, in
which members of $X$ may be
thought of as black points, and members of $\Z^n \setminus X$
as white points, of a picture of some ``real world" 
object or scene.

\subsection{Adjacencies}
Let $u,n \in \N$, $1 \le u \le n$. 
The most-used adjacencies
in digital topology are the
$c_u$ adjacencies, especially $c_1$ and $c_n$.
These are defined as
follows. Let $x,y \in \Z^n$,
\[ x = (x_1, \ldots, x_n),~~~~~y = (y_1,\ldots, y_n). \]
We say $x$ and $y$ are 
{\em $c_u$-adjacent} if
\begin{itemize}
    \item $x \neq y$, and
    \item for at most $u$ indices~$i$, 
          $\mid x_i - y_i \mid = 1$, and
    \item for all indices $j$ such that 
          $\mid x_j - y_j \mid \neq 1$, we have
          $x_j = y_j$.
\end{itemize}

In low dimensions, many 
authors denote a
$c_u$ adjacency by the number of points that can
have this adjacency with a given point in $\Z^n$. E.g.,
\begin{itemize}
    \item In $\Z^1$, $c_1$-adjacency is 2-adjacency.
    \item In $\Z^2$, $c_1$-adjacency is 4-adjacency and
          $c_2$-adjacency is 8-adjacency.
    \item In $\Z^3$, $c_1$-adjacency is 8-adjacency,
          $c_2$-adjacency is 18-adjacency, and
          $c_3$-adjacency is 26-adjacency.
\end{itemize}

We use the notations $y \adj_{\kappa} x$, or, when
the adjacency $\kappa$ can be assumed, $y \adj x$, to mean
$x$ and $y$ are $\kappa$-adjacent.
The notations $y \adjeq_{\kappa} x$, or, when
$\kappa$ can be assumed, $y \adjeq x$, mean either
$y=x$ or $y \adj_{\kappa} x$.

Let
\[
N(X,x,\kappa) = \{ \, y \in X \mid y \adjeq_{\kappa} x \, \}.
\]

A sequence $P=\{y_i\}_{i=0}^m$ in a digital image $(X,\kappa)$ is
a {\em $\kappa$-path from $a \in X$ to $b \in X$} if
$a=y_0$, $b=y_m$, and $y_i \adjeq_{\kappa} y_{i+1}$ 
for $0 \leq i < m$. 
We say~$m$ is the 
{\em length} of the path:
\[ m = length(P).\]

$X$ is {\em $\kappa$-connected}~\cite{Rosenfeld},
or {\em connected} when $\kappa$
is understood, if for every pair of points $a,b \in X$ there
exists a $\kappa$-path in $X$ from $a$ to $b$. 

The metric used throughout this paper is
the following.

\begin{definition}
{\rm \cite{ChartTian}}
Let $(X,\kappa)$ be a connected graph. The 
{\em shortest path metric} for $(X,\kappa)$ is
\[ d(x,y) = \min\{length(P) ~|~ P 
   \mbox{ is a $\kappa$-path in $X$ from $x$ to $y$}\},
   \mbox{ for } x,y \in X.
\]
\end{definition}
When it is useful to clarify
the digital image considered,
we use the notation
$d_{(X,\kappa)}(x,y)$ for
this metric.

\subsection{Digitally continuous functions}
Digital continuity preserves connectedness, as at
Definition~\ref{continuous}. By
using adjacency as our notion 
of ``closeness," we get Theorem~\ref{continuityPreserveAdj}.

\begin{definition}
\label{continuous}
{\rm ~\cite{Boxer99} (generalizing a definition of~\cite{Rosenfeld})}
Let $(X,\kappa)$ and $(Y,\lambda)$ be digital images.
A function $f: X \rightarrow Y$ is 
$(\kappa,\lambda)$-continuous if for
every $\kappa$-connected $A \subset X$ we have that
$f(A)$ is a $\lambda$-connected subset of $Y$.
\end{definition}

If $Y \subset X$, we use the abbreviation
{\em $\kappa$-continuous} for 
{\em $(\kappa,\kappa)$-continuous}.

When the adjacency relations are understood, we will simply say that $f$ is \emph{continuous}. Continuity can be expressed in terms of adjacency of points:

\begin{thm}
{\rm ~\cite{Rosenfeld,Boxer99}}
\label{continuityPreserveAdj}
A function $f: X \to Y$ is
$(\kappa,\lambda)$-continuous 
if and only if
$x \adj_{\kappa} x'$  
implies $f(x) \adjeq_{\lambda} f(x')$.
\end{thm}

See also~\cite{Chen94,Chen04}, where similar notions are referred to as {\em immersions}, {\em gradually varied operators},
and {\em gradually varied mappings}.

A digital {\em isomorphism} (called {\em homeomorphism}
in~\cite{Boxer94}) is a $(\kappa,\lambda)$-continuous
surjection $f: X \to Y$ such that $f^{-1}: Y \to X$ is
$(\lambda,\kappa)$-continuous.

The literature uses {\em path} polymorphically: a
$(c_1,\kappa)$-continuous function $f: [0,m]_{\Z} \to X$
is a $\kappa$-path if $f([0,m]_{\Z})$ is a 
$\kappa$-path from $f(0)$ 
to $f(m)$ as described above.

We define~\cite{bs20}
\[ C(X,\kappa) =
\{ \, f: X \to X \mid f
   \mbox{ is $\kappa$-continuous} \, \}. 
\]

Let $X \subset \Z^n$. The {\em projection functions}
$p_i: X \to \Z$ are defined by
\[ p_i(x_1, \ldots, x_i, \ldots, x_n) = x_i.
\]
The projection functions
are $(c_u,c_1)$-continuous
for $1 \le u \le n$~\cite{Han05}.

For a function $f: X \to X$, we denote by
$\Fix(f)$ the set of fixed points of $f$, i.e.,
$x \in \Fix(f)$ means $f(x)=x$.

\subsection{Freezing and cold sets}
Material in this section is largely quoted or paraphrased from~\cite{BxFpSets2}.

In a Euclidean space, knowledge of the fixed point set of a continuous self-map
$f: X \to X$ often gives little information about $f|_{X \setminus \Fix(f)}$. By contrast,
knowledge of $\Fix(f)$ for $f \in C(X,\kappa)$ can tell us much about $f|_{X \setminus \Fix(f)}$.
This motivates the study of freezing and cold sets.

\begin{definition}
\label{freezeDef}
{\rm \cite{BxFpSets2}}
Let $(X,\kappa)$ be a digital image. We say $A \subset X$ is a 
{\em freezing set for $(X,\kappa)$} if given $g \in C(X,\kappa)$,
$A \subset \Fix(g)$ implies $g=\id_X$. If no proper subset of a freezing set $A$ 
is a freezing set for $(X,\kappa)$, then $A$ is
a {\em minimal freezing set} for $(X,\kappa)$.
\end{definition}

\begin{thm}
{\rm \cite{BxFpSets2}}
\label{corners-min}
Let $X = \Pi_{i=1}^n [0,m_i]_{\Z}$.
Let $A = \Pi_{i=1}^n \{0,m_i\}$.
\begin{itemize}
\item Let $Y = \Pi_{i=1}^n [a_i,b_i]_{\Z}$ be
      such that $X \subset Y$. Let $f: X \to Y$ be
      $c_1$-continuous. If $A \subset \Fix(f)$, then $X \subset \Fix(f)$.
\item $A$ is a freezing set for $(X,c_1)$; minimal for $n \in \{1,2\}$.
\end{itemize}
\end{thm}

\begin{thm}
{\rm \cite{BxFpSets2}}
\label{freezeInvariant}
Let $A$ be a freezing set for the digital image $(X,\kappa)$ and let
$F: (X,\kappa) \to (Y,\lambda)$ be an isomorphism. Then $F(A)$ is
a freezing set for $(Y,\lambda)$.
\end{thm}

\subsection{Tools for identifying freezing sets}
The following are useful for determining fixed 
point and freezing sets.

\begin{prop}
{\rm (Corollary 8.4 of~\cite{bs20})}
\label{uniqueShortestProp}
Let $(X,\kappa)$ be a digital image and
$f \in C(X,\kappa)$. Suppose
$x,x' \in \Fix(f)$ are such that
there is a unique shortest
$\kappa$-path $P$ in~$X$ from $x$ 
to $x'$. Then $P \subseteq \Fix(f)$.
\end{prop}

Lemma~\ref{c1pulling}, below,
\begin{quote}
$\ldots$ can
be interpreted to say that
in a $c_u$-adjacency,
a continuous function that
moves a point~$p$ also moves
a point that is ``behind"
$p$. E.g., in $\Z^2$, if $q$ and $q'$ are
$c_1$- or $c_2$-adjacent with $q$
left, right, above, or below $q'$, and a
continuous function $f$ moves $q$ to the left,
right, higher, or lower, respectively, then
$f$ also moves $q'$ to the left,
right, higher, or lower, respectively~\cite{BxFpSets2}.
\end{quote}

\begin{lem}
\label{c1pulling}
{\rm ~\cite{BxFpSets2}}
Let $(X,c_u)\subset \Z^n$ be a digital image, 
$1 \le u \le n$. Let $q, q' \in X$ be such that
$q \adj_{c_u} q'$.
Let $f \in C(X,c_u)$.
\begin{enumerate}
    \item If $p_i(f(q)) > p_i(q) > p_i(q')$
          then $p_i(f(q')) > p_i(q')$.
    \item If $p_i(f(q)) < p_i(q) < p_i(q')$
          then $p_i(f(q')) < p_i(q')$.
\end{enumerate}
\end{lem}

We extend Lemma~\ref{c1pulling} to digital
arcs as follows.

\begin{cor}
\label{pullingSegment}
    Let $1 \le u \le n$ and let
    $(R,c_u) \subset (\Z^n,c_u)$ be a
    digital arc, i.e., for some integer
    $k > 0$ there is an isomorphism
    $g: ([0,k]_{\Z}, c_1) \to (R, c_u)$.
    Let $f: R \to \Z^n$ be $c_u$-continuous.
    \begin{itemize}
        \item If $g$ is monotone 
    increasing in the $i^{th}$ coordinate, 
    i.e., $p_i(j+1) > p_i(j)$ for $0 \le j < k$,
    and $p_i(f(g(m))) > p_i(g(m))$ for some
    $m$ such that $0 < m \le k$, then
    $p_i(f(g(j))) > p_i(g(j))$ for 
    $0 \le j \le m$.
    \item If $g$ is monotone 
    decreasing in the $i^{th}$ coordinate, 
    i.e., $p_i(j+1) < p_i(j)$ for $0 \le j < k$,
    and $p_i(f(g(m))) < p_i(g(m))$ for some
    $m$ such that $0 < m \le k$, then
    $p_i(f(g(j))) < p_i(g(j))$ for 
    $m \le j \le k$.
    \end{itemize}
\end{cor}

\begin{proof}
    These assertions follow from elementary
    inductions based on Lemma~\ref{c1pulling}.
\end{proof}

We are grateful to an anonymous referee for
suggesting the following.

\begin{prop}
\label{nbrPropertyForEssentialFreezePt}
    Let $(Y,\kappa)$ be a digital image. Let
    $y \in Y$. If there exists 
    $y' \in Y \setminus \{y\}$ such that
    $N(Y,y,\kappa) \subset N(Y,y',\kappa)$, 
    then ~$y$ belongs to every freezing set 
    of~$(Y,\kappa)$.
\end{prop}

\begin{proof}
    Let $A$ be a freezing set for $(Y,\kappa)$.
    Let $y',w \in Y$ be as stated above.
    Suppose for a contradiction that 
    $y \not \in A$. Let $f: Y \to Y$ be
    given by
    \[ f(x) = 
       \left \{ \begin{array}{ll}
           x &  \mbox{if } x \neq y; \\
           w &  \mbox{if } x = y.
       \end{array}
       \right .
    \]
    By our choices of $y,y',w$, it follows
    that $f \in C(Y,\kappa)$. Further,
    $f|_A = \id_A$, but $f \neq \id_Y$,
    contrary to~$A$ being a freezing set.
    Thus~$y \in A$.
\end{proof}

\begin{definition}
\label{bdDef}
{\rm \cite{RosenfeldMAA}}
Let $X \subset \Z^n$. The
{\em boundary of} $X$ is
\[Bd(X) = \{x \in X \, | \mbox{ there exists } y \in \Z^n \setminus X \mbox{ such that } y \adj_{c_1} x\}.
\]
\end{definition}

\begin{thm}
{\rm \cite{BxFpSets2}}
\label{bdFreezes}
Let $X \subset \Z^n$ be finite. Then for 
$1 \le u \le n$, $Bd(X)$ is 
a freezing set for $(X,c_u)$.
\end{thm}

Moreover, we have the following.

\begin{thm}
{\rm \cite{BxFpSets2}}
\label{noProperSub}
Let $X = \prod_{i=1}^n[0,m_i]_{\Z}  \subset \Z^n$, where $m_i>1$ for all~$i$.
Then $Bd(X)$ is a minimal freezing set for $(X,c_n)$.
\end{thm}

Other tools for computing freezing sets
are studied in~\cite{BxTrimBld}.

\subsection{Cold sets}
\begin{definition}
\label{s-cold-def}
{\rm \cite{BxFpSets2}}
Given $s \in \N^*$, we say $A \subset X$ is an
{\em $s$-cold set} for the connected digital image $(X,\kappa)$
if given $f \in C(X,\kappa)$ such that
$f|_A = \id_A$, then for all $x \in X$, $d_{(X,\kappa)}(x,f(x)) \le s$.
If no proper subset of $A$ is an $s$-cold set
for $(X,\kappa)$, then $A$ is {\em minimal}.
A {\em cold set} is a 1-cold set.
\end{definition}

Note a freezing set is a 0-cold set.

\begin{thm}
\label{s-cold-invariant}
{\rm \cite{BxFpSets2}}
Let $(X,\kappa)$ be a connected digital image, 
let $A$ be an $s$-cold set for $(X,\kappa)$,
and let $F: (X,\kappa) \to (Y,\lambda)$ be an
isomorphism. Then $F(A)$ is
an $s$-cold set for $(Y,\lambda)$.
\end{thm}

\begin{definition}
\label{n-map}
{\rm \cite{BxLtd}}
Let $X$ be a $\kappa$-connected digital image. 
Let $f \in C(X,\kappa)$ 
and let $A \subset X$. Let $m,n \in \N^*$.
If for all $x \in A$ we have 
$d_{(X,\kappa)}(x,f(x)) \le n$, we say
$f|_A$ is an $n$-map. We write ``$f$ is an $n$-map" 
if $f|_X$ is an $n$-map.
\end{definition}

If for all $f \in C(X,\kappa)$, $f|_A$ being 
an $m$-map implies $f$ is an $n$-map, then $A$ is an
{\em $(m,n)$-limiting set for} $(X,\kappa)$
and $(X,\kappa)$ is {\em $(A,m,n)$-limited}.

\subsection{Cones and suspensions}
The paper~\cite{LuptonEtal}
introduces cone and suspension
constructions in digital
topology, according to the
following.

\begin{definition}
Let $(X,\kappa)$ be a 
digital image.
    \begin{itemize}
        \item The {\em cone on} $X$ is the set
        $CX = X \cup \{U\}$, where $U \not \in X$, 
        and adjacency $\kappa'$, where $U \adj_{\kappa'} x$ for
        all $x \in X$ and if $a,b \in X$ then
        $a \adj_{\kappa'} b$ if and only if
        $a \adj_{\kappa} b$.
    \item The {\em 
     suspension              
        on} $X$ is the set
        $SX = X \cup \{U, L\}$,
        where $U \neq L$ and
        $\{U, L\} \cap X = \emptyset$; 
        and adjacency $\kappa'$,
        where $U \adj_{\kappa'} x
        \adj_{\kappa'} L$ for
        all $x \in X$; $U$ and $L$ are not adjacent; and if
        $a,b \in X$ then
        $a \adj_{\kappa'} b$
        if and only if
        $a \adj_{\kappa} b$.
    \end{itemize}
\end{definition}
One may consider $U$ as the
``upper pole" of $CX$ and
of $SX$, and $L$ as
the ``lower pole" of $SX$.

When $\kappa$ can be
assumed, we often abbreviate
$(CX,\kappa')$ as $CX$, and
$(SX,\kappa')$ as $SX$.

For both $(CX,\kappa')$
and $(SX,\kappa')$, we
implicitly use the following.

\begin{thm}
    {\rm \cite{LuptonEtal}}
    Any finite graph 
    $(X,\kappa)$ can be 
    isomorphically mapped 
    onto a subset of
    $(Z^{n-1},c_{n-1})$, where $n = \#X$.
\end{thm}

\section{Freezing and cold sets for $CX$ and $SX$}
\begin{remark}
\label{d(U,L)}
$d_{SX}(U,L)=2$.
\end{remark}

Given a graph $(X,\kappa)$ and
$D \subset X$, we say~$D$
is a {\em $\kappa$-dominating set} in~$X$ if
for every $x \in X$ there
exists $d \in D$ such that
$d \adjeq_{\kappa} x$.

\begin{prop}
\label{limitDominating}
    Let $(X,\kappa)$ be a connected digital image. 
    Let $f \in C(X,\kappa)$.
    Let $D \subset X$ such that $D$ is a $\kappa$-dominating
    set in $X$. If $f|_D$ is an $m$-map, then 
    $f$ is an $(m+2)$-map.
\end{prop}

\begin{proof}
Assume $D$ is a $\kappa$-dominating set in $X$.
    For each $x \in X$, there exists $\delta_x \in D$ such that
    $x \adjeq_{\kappa} \delta_x$. By continuity
    \begin{equation}
    \label{densityIneq}
        d_{(X,\kappa)}(x,f(x)) \le d(x, \delta_X) +
        d(\delta_X, f(\delta_X)) + 
         d(f(\delta_X), f(x))
    \end{equation}  
    \[ \le 1 + m + 1 = m+2. 
    \]
    The assertion follows.
\end{proof}

\begin{prop}
\label{distAtMost2}
Let $(X,\kappa)$ be a digital image.
Given $x,y \in CX$, $a,b \in SX$,
we have $d_{CX}(x,y) \le 2$, $d_{SX}(a,b) \le 2$.
\end{prop}

\begin{proof}
We must show that there are paths of length at most~2 from~$x$ to~$y$ in~$CX$
and from~$a$ to~$b$ in~$SX$.

Assume without loss of generality that $x \neq y$.
For $\{x,y\} \subset CX$, $\{x, U, y\}$ is a path of length at most 2.

In $SX$, for $\{a,b\} \subset X$,
$\{a, U, b\}$ is a path in $SX$ from $a$ 
to $b$ of length 2. 
For $a \in X$, $\{a,L\}$ is a path in $SX$
of length~1 from~$a$ to~$L$.
For $b \in X$, $\{U,b\}$ is a path in $SX$
of length~1 from~$U$ to~$b$.
For any $x \in X$,
$\{U,x,L\}$ is a path in $SX$ from $U$ to $L$
of length~2. The assertion follows.
\end{proof}

\begin{remark}
    In light of Proposition~\ref{distAtMost2}, 
    for every digital image $(X,\kappa)$, every
    nonempty $A \subset CX$,
    every nonempty $A' \subset SX$ and $m \ge 0$,
    $A$ is $(m,2)$-limiting for $CX$ and $A'$
    is $(m,2)$-limiting for $SX$. In this
    paper, for both $CX$
    and $SX$, we mostly study $(0,0)$-limiting  
    (freezing) and $(0,1)$-limiting (cold) sets.
\end{remark}

\begin{thm}
\label{coneFreezeThm}
    Let $(X,\kappa)$ be a digital image such that
    \begin{equation}
    \label{XnotSmall}
        \mbox{for every $x \in X$, $N(X,x,\kappa) \neq X$}
    \end{equation}
    Then $X$ is a minimal freezing set for $CX$.
\end{thm}

\begin{proof}
    Since $U$ is adjacent to every point of~$X$,
    it follows from     Proposition~\ref{nbrPropertyForEssentialFreezePt} that
    \begin{equation} 
    \label{XinAllFreezers}
    X \mbox{ is a subset of every freezing set for } CX.
    \end{equation}

    The function $f: CX \to CX$ given, for some
    $y \in Y$, by
    \[ f(x) = \left \{ \begin{array}{ll}
       x & \mbox{if } x \in X, \\
       y & \mbox{if } x = U,
    \end{array} \right .
    \]
    cannot be continuous, by~(\ref{XnotSmall}). Therefore,
    if a function $g: CX \to CX$ is continuous and
    $g|_X = \id_X$, we     must have $g(U) = U$, so 
    $g = \id_{CX}$. Therefore, $X$ is a freezing set for $CX$. 
    It follows from~(\ref{XinAllFreezers}) that $X$ is minimal.   
\end{proof}

\begin{prop}
    \label{polesInFreezeSX}
    Let $B$ be a freezing set, a cold set.
    or a $(1,1)$-limiting set for $SX$. Then
    $\{ \, U, L \, \} \subset B$.
\end{prop}

\begin{proof}
Let $\{y,y'\} = \{U,L\}$.
    If $y \not \in B$, then the function $f: SX \to SX$ given by
    \[ f(x) = \left \{ \begin{array}{ll}
       x  & \mbox{if } x \neq y, \\
       y' &  \mbox{if } x = y
    \end{array}
    \right .
    \]
    is easily seen to be continuous and satisfies $f|_B = \id_B$. However, 
    $d(y,f(y)) = 2$, contrary to our
    assumption about~$B$. Thus we must have $y,y' \in B$.
\end{proof}

\begin{lem}
    \label{polesToPolesKeepX}
    Let $f \in C(SX,\kappa')$ such that
    $f|_{\{ \, U, L \, \}} = \id_{\{ \, U, L \, \}}$.
    Then $f(X) \subset X$.
\end{lem}

\begin{proof}
    Suppose otherwise. Then for 
some $x \in X$, either $f(x) = U$ or $f(x)=L$. In the former
case, we have $x \adj L$; but $f(x) = U$ and
$f(L)=L$ are neither adjacent nor equal, contrary to
the continuity of~$f$. In the latter case, $x \adj U$;
but $f(x) = L$ and $f(U)=U$ are neither
adjacent nor equal, contrary to the continuity of~$f$.
These contradictions establish the assertion.
\end{proof}

\begin{thm}
    Let $A \subset (X,\kappa)$ and
    $B = A \cup \{ \, U, L \, \} \subset SX$.
    Then $A$ is a (minimal) 
    freezing set for $(X,\kappa)$ if
    and only if $B$ is a (minimal) freezing set for $SX$.
\end{thm}

\begin{proof}
Suppose $A$ is a freezing set for $(X,\kappa)$.
Let $f: SX \to SX$ be continuous, such that $f|_B = \id|_B$. By Lemma~\ref{polesToPolesKeepX}, $f|_X: X \to X$; and
$f|_A = \id_A$. By choice of $A$, $f|_X = \id_X$. Thus
$f = \id_{SX}$, so $B$ is freezing.

Suppose $A$ is a minimal freezing set for $(X,\kappa)$. Suppose for a contradiction
that~$B$ is not minimal. Then there exists
a proper subset~$B'$ of~$B$ such that $B'$ is a freezing set for $(SX,\kappa')$.
Therefore
\begin{equation} 
\label{ifBnotMinimal}
U \not \in B' \mbox{ or } L \not \in B'
   \mbox{ or there exists } a \in A \setminus B'. 
\end{equation}
But by Proposition~\ref{polesInFreezeSX},
$\{U,L\} \subset B'$. Thus the only
remaining possibility 
from~\ref{ifBnotMinimal} is that there 
exists $a \in A \setminus B'$. But then
we see easily that the function
$g: SX \to SX$ given by
\[ g(x) = \left \{
\begin{array}{ll}
   x & \mbox{if } x \neq a; \\
   U  & \mbox{if } x = a,
\end{array}
\right .
\]
satisfies $g \in C(SX, \kappa')$ and $g|_{B'} = \id_{B'}$. But $g \neq \id_{SX}$,
contrary to the assumption that $B'$ is 
freezing. We must therefore have that~$B$
is a minimal freezing set for $(SX,\kappa')$.

Let $B$ be a freezing set for $SX$. By 
Proposition~(\ref{polesInFreezeSX}), 
\[
    \{ \, U, L \, \} \subset B
\]
Let $f \in C(X,\kappa)$
such that $f|_A = \id_A$.
Then the function 
$g: SX \to SX$ given by
\[ g(y) = \left \{ \begin{array}{ll}
    f(y) & \mbox{if } y \in X, \\
    y & \mbox{if } y \in
        \{U,L\},
\end{array}
\right .
\]
is easily seen to belong to
$C(SX,\kappa')$, with $g|_B = \id_B$.
By choice of $B$, 
$g = \id_{SX}$. Therefore,
$f = g|_X = \id_X$.
Thus $A$ is a freezing set
for $(X,\kappa)$.

Suppose $B$ is a minimal freezing set
for~$(SX,\kappa')$.
If a subset~$A'$ of $A$ is also a freezing set for~$X$, by the above and
Proposition~\ref{polesInFreezeSX}, 
\[ B' = A' \cup \{U,L\} \subset A \cup \{U,L\} = B
\]
is a freezing set
for~$CX$. Since $B' \subset B$ and~$B$ is minimal for~$CX$,
we must have $B'=B$, hence $A'=A$. Hence $A$ is minimal for~$X$.
\end{proof}

\begin{prop}
    Let $(X,\kappa)$ be a connected digital image. 
    Let $Y$ be a cold set or a
    $(1,1)$-limiting set for $(X,\kappa)$.
    Let $f \in C(CX,\kappa')$ such that $f|_Y=\id_Y$ (respectively, $f|_Y$ is a $1$-map)
    and $f(X) \subset X$. Then~$f$ is a $1$-map.
\end{prop}

\begin{proof}
 It follows immediately in both cases
 that $f|_Y: Y \to X$ is a $1$-map,
 and therefore $f|_X$ is a $1$-map. Since 
 $d(U,f(U)) \le 1$, the assertion follows.
\end{proof}

\begin{thm}
    Let $(X,\kappa)$ be a digital image. Then 
    $U$ is not a member of a
minimal (1,1)-limiting subset $A$ of $CX$.
\end{thm}

\begin{proof}
    Suppose $U \in A$ where $A$ is (1,1)-limiting for $CX$.
    Let $A' = A \setminus \{U\}$. Let $f \in C(CX,\kappa')$ such that
    $f|_{A'}$ is a 1-map. Since $f|_{\{U\}}$ must be a 1-map,
    it follows that $f|_A$ is a 1-map, so by choice of $A$,
    $f$ is a 1-map. Thus $A'$ is (1,1)-limiting for $CX$;
    therefore $A$ is not minimal.
\end{proof}

\begin{thm}
    Let $(X,\kappa)$ be a digital image with
    property~(\ref{XnotSmall}). Then $U$ and $L$ are members of
    every $A \subset SX$ such that $SX$ is $(A,1,1)$-limited.
\end{thm}

\begin{proof}
 Let $A \subset X$ be such that 
 \begin{equation}  
    \label{SX-A,1,1}
     SX \mbox{ is } (A,1,1)-\mbox{limited.}
 \end{equation}
 We argue by contradiction. Suppose $U \not \in A$. 
 Let $f: SX \to SX$ be the function
\[ f(x) = \left \{ \begin{array}{ll}
    x & \mbox{if } x \neq U, \\
    L &  \mbox{if } x = U.
\end{array} \right .
\]
It is easily seen that $f \in C(SX,\kappa')$ and $f|_A$ is a 1-map,
but by Remark~\ref{d(U,L)}, $d_{SX}(U, f(U)) = 2$, contrary 
to~(\ref{SX-A,1,1}). Thus we must have $U \in A$.

A similar argument shows we must have $L \in A$.
\end{proof}

\section{Limiting sets for digital
pyramids}
In this section, we consider 
freezing and cold sets for other
cone-like digital images. In 
particular, for $n \in \N$, we
consider a digital version of the
surface of a pyramid 
(see Figures~\ref{fig:pyramid}
and~\ref{fig:T_1andT_2})
\[ P_n = \bigcup_{i=0}^n T_i 
  \subset \Z^3
\]
where 
 \[ T_i = ( \{ \, -i,i \, \} \times [-i,i]_{\Z}  \times \{n-i\})
    \cup ([-i,i]_{\Z}  \times \{ \, -i,i \, \}  \times \{n-i\}).
    \]
    Also, let $T_i'$ be the set of ``corner" points
    of~$T_i$:
\[ T_i' = \{ \, (-i,-i,n-i), (i,-i,n-i),  (i,i,n-i), (-i,i,n-i) \, \}
   \subset T_i,
    \]
    \[W_i = [-i,i]_{\Z}^2 
    \times  \{n-i\}.
    \]
    Figure~\ref{fig:T_1andT_2} illustrates that in
    the $c_3$ adjacency, each $T_i$ is a 
    digital square (without interior)
    whose ``corners" are the members
    of $T_i'$; the two $c_1$-neighbors in $T_i$ of
    each ``corner" are $c_2$-adjacent.
$P_n$ can be thought of as a digital analog of 
a cone on a square (without interior).

Let
\begin{equation}
    \label{solidPyramidDef}
    Q_n = \bigcup_{i=0}^n W_i \subset \Z^3,
\end{equation}
\[ 
\]
which can be thought of as a 
digital analog of a solid 
tetrahedron.

For each of $P_n,Q_n$, we let
$U = (0,0,n)$, the high point of the image.

\begin{figure}
    \centering
    \includegraphics[height=2in]{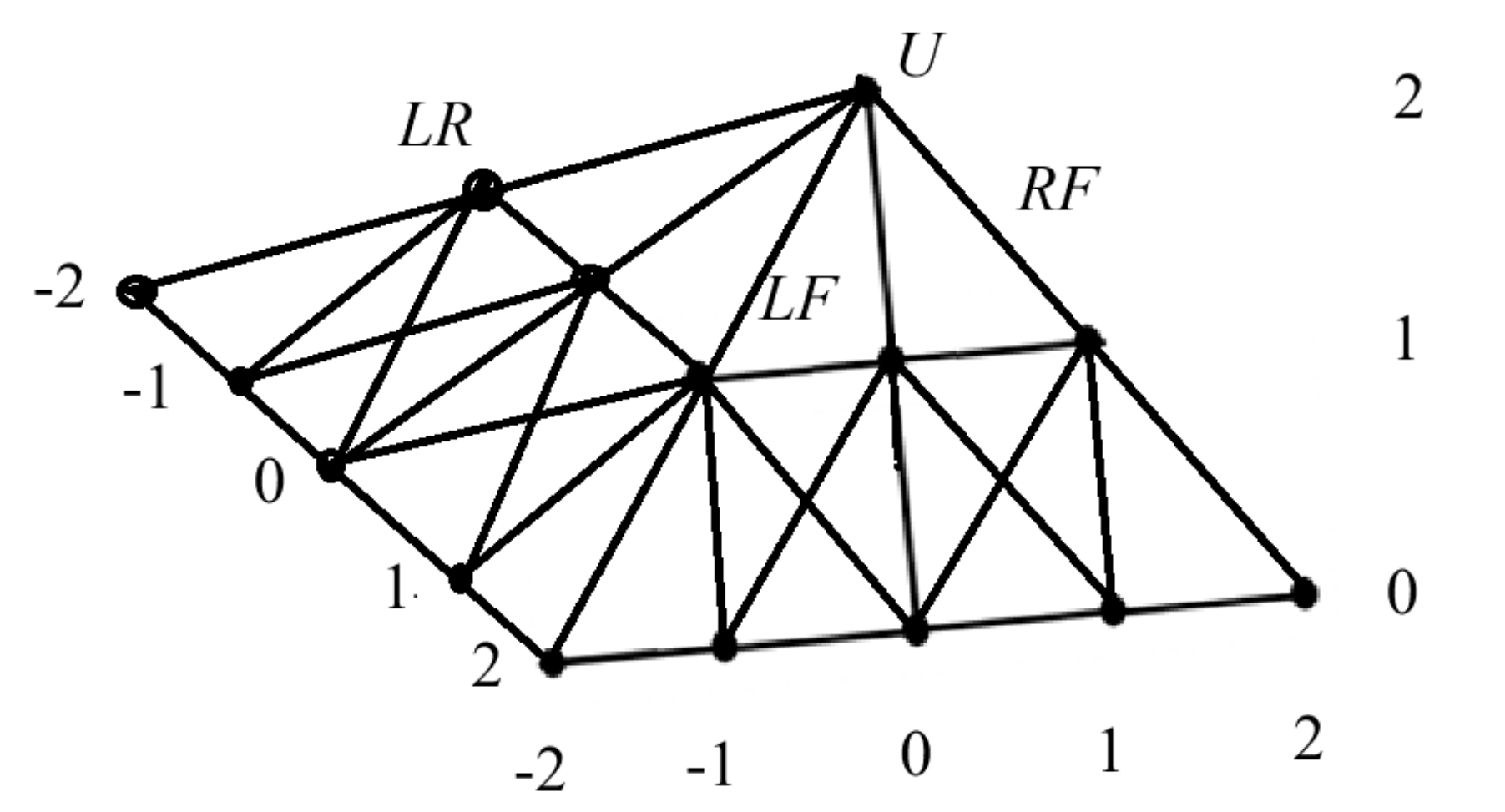}
    \caption{Left and front 
    sides of the digital 
    pyramid $(P_2,c_3)$
    with base
    $[-2,2]_{\Z}^2$. \newline
    Marked edges: \newline
    {\em LR} (left rear) from (-2,-2,0) to {\em U} \newline
    {\em LF} (left front) from (2,-2,0) to {\em U} \newline
    {\em RF} (right front) from (2,2,0) to {\em U} \newline
    Faces shown: \newline
    {\em L} (left) - the digital triangle with vertices $U$,
                     $(-2,-2,0)$, and $(2,-2,0)$ \newline
    {\em F} (front) - the digital triangle with vertices $U$,
                     $(2,-2,0)$, and $(2,2,0)$
    }
    \label{fig:pyramid}
\end{figure}

We will refer to the edges of the pyramid containing $U$ as follows.
\newline $LR$ (left rear): the digital segment from $(-n,-n,0)$ to $U$
\newline $LF$ (left front): the digital segment from $(n,-n,0)$ to $U$
\newline $RF$ (right front): the digital segment from $(n,n,0)$ to $U$
\newline $RR$ (right rear): the digital segment from $(-n,n,0)$ to $U$
\newline $BL$ (bottom left): the digital segment from $(-n,-n,0)$ to
$(n,-n,0)$
\newline $BF$ (bottom front): the digital segment from $(n,-n,0)$ to
$(n,n,0)$
\newline $BR$ (bottom right): the digital segment from $(n,n,0)$ to
$(-n,n,0)$
\newline $BB$ (bottom back): the digital segment from $(-n,n,0)$ to
$(-n,-n,0)$ \newline

 Faces of the pyramid are the digital triangles as follows:
\newline $L$ (left): edges are $BL, LR, LF$
\newline $F$ (front): edges are $BF, LF, RF$
\newline $R$ (right): edges are $BR, RF, RB$
\newline $B$ (back): edges are $BB, RB, LB$

\begin{figure}
    \includegraphics[height=2in]{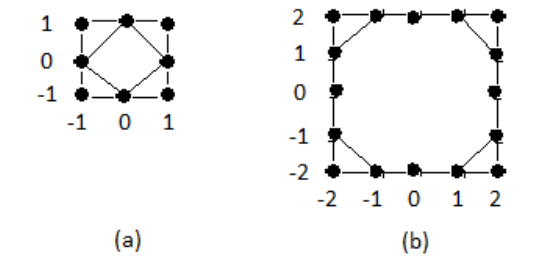}
    \caption{Slices of $(P_n,c_3)$ \newline
    (a) an isomorphic copy in $(\Z^2,c_2)$ of $(T_1,c_3)$ \newline 
          (b)  an isomorphic copy in $(\Z^2,c_2)$ of $(T_2,c_3)$
          }
    \label{fig:T_1andT_2}
\end{figure}

\subsection{Freezing sets for $P_n$}
\label{FreezeP_nSection}
It is easily seen (although perhaps easily overlooked) that
$(P_n,c_1)$ is not connected,
as each $T_i$ is a 
$c_1$-component. Therefore, we consider $(P_n,c_3)$.

We use the following.

\begin{lem}
    \label{pyramidLikeSquare}
$(P_n,c_3)$ is isomorphic to $[-n,n]_{\Z}^2$.
\end{lem}

\begin{proof}
        The function $F: P_n \to 
[-n,n]_{\Z}^2$ given by
$F(a,b,c) = (a,b)$ is a
$(c_3,c_2)$-isomorphism. This
can be seen by imagining 
looking down on $P_n$ from
above~$U$. Note
\[ F(T_n) =
([-n,n]_{\Z} \times \{-n,n \}) \cup (\{-n,n \} \times [-n,n]_{\Z}) = 
Bd([-n,n]_{\Z}^2).
\]
\end{proof}

\begin{thm}
\label{uniqueRectangleFreezer}
    Let $X = [a_1,b_1]_{\Z} \times [a_2,b_2]_{\Z}$. Then $Bd(X)$ is the unique 
    minimal freezing set for $(X,c_2)$.
\end{thm}

\begin{proof}
This follows from Theorem~\ref{noProperSub}.
\end{proof}

\begin{thm}
\label{P_nMinFreeze}
    $T_n$ is the unique minimal freezing set for $P_n$.
\end{thm}

\begin{proof}
This follows from 
Theorem~\ref{freezeInvariant},
Theorem~\ref{noProperSub}, and
Theorem~\ref{uniqueRectangleFreezer}.
\end{proof}

\subsection{Freezing sets for digital
solid pyramids}
In this section, we consider
digital images $(Q_n,c_3)$, where 
$Q_n$ is as defined 
at~(\ref{solidPyramidDef}).
Notice $Q_n$ is a cone-like structure
over a base of the (solid) digital
square $W_n$; $Q_n$ may be regarded as
a solid version of $P_n$.

As above, we let $U$ be the 
highest point of $Q_n$:
$U = (0,0,n)$.

We have the following.

\begin{lem}
\label{W_nInFreezingSet}
    $\{U\} \cup W_n$ is a subset of every
    freezing set for $(Q_n,c_3)$.
\end{lem}

\begin{proof}
From Proposition~\ref{nbrPropertyForEssentialFreezePt}, members
listed under the following ``for~$y$" column belong to
every freezing set of $(Q_n,c_3)$.
\begin{tabbing}
     \underline{for $y$}~~~~~~~~~~~~~~~~~~~~~~~~~~~~~~~~~~~\=
        \underline{and for $y'$} \\
        $U = (0,0,n)$ \> $(0,0,n-1)$ \\
         Members of $T_i'$: \\ 
        $(-n,-n,0)$ \> $(-(n-1),-(n-1),0)$ \\
        $(n, -n, 0)$ \> $(n-1, -(n-1), 0)$ \\
        $(n,n,0)$ \> $(n-1,n-1,0)$ \\
        $(-n,n,0)$ \> $(-(n-1), n-1, 0)$ \\
         ``Interiors" of ``edges" of $W_n$: \\
        $(a,-n,0) \in L$, $|a| < n$ \> $(a,-(n-1),0)$ \\
        $(n,b,0) \in F$, $|b| < n$ \> $(n-1,b,0)$ \\
        $(a,n,0) \in R$,  $|a| < n$ \> (a, n-1, 0) \\
        $(-n,b,0) \in B$, $|b| < n$ \> (-(n-1),b,0) \\
        ``Interior" points of $W_n$: \\
        $(a,b,0)$, $|a| < n$, $|b| < n$ \> $(a,b,1)$
\end{tabbing}
\end{proof}

\begin{thm}
\label{minFreezeForQ_n}
    $A=\{U\} \cup W_n$ is the unique minimal
    freezing set for $(Q_n,c_3)$.
\end{thm}

\begin{proof}
    Let $f \in C(Q_n,c_3)$ such that $f|_A = \id_A$.

As each of $LR,LF,RF,RR$ is the unique shortest $c_3$-path
between its endpoints, by Proposition~\ref{uniqueShortestProp},
\begin{equation}
\label{vertEdgesFixed}
    LR \cup LF \cup RF \cup RR \subset \Fix(f).
\end{equation}

Next, we show $L \subset \Fix(f)$. Let 
\[ S_k = \{(j,k,n-k) \mid -k \le j \le k\} \subset L.
\]
This is a digital segment in~$L$
\begin{equation}
\label{endptsFixed}
    \mbox{with endpoints $(-k,-k,n-k),~(k,-k,n-k)$ in
$LR \cup LF \subset \Fix(f)$.}
\end{equation}
 Suppose for a contradiction
that for some $x \in S_k$ we have $x \not \in \Fix(f)$.
Then, for some index $m \in \{1,2,3\}$,
$p_m(f(x)) \neq p_m(x)$.
\begin{itemize}
    \item Suppose $p_1(f(x)) > p_1(x)$. Then by 
    Corollary~\ref{pullingSegment}, (at the ``back"
    endpoint of $S_k$,) $p_1(f(-k,-k,n-k)) > -k$,
    contrary to~(\ref{endptsFixed}).
    \item Suppose $p_1(f(x)) < p_1(x)$. Then by 
    Corollary~\ref{pullingSegment}, (at the ``front"
    endpoint of $S_k$,) $p_1(f(k,-k,n-k)) < k$,
    contrary to~(\ref{endptsFixed}).
    \end{itemize}
Thus we must have 
\begin{equation}
\label{p_1InQ_n}
    p_1(f(x)) = p_1(x).
\end{equation}

\begin{itemize}
    \item Suppose $p_2(f(x)) > p_2(x)$. Without loss of
    generality, $x=(a,-k,n-k)$ where
    \[ a = \min_{y \in S_k} \{p_2(y) \mid p_2(f(y)) \neq p_2(y) \}.
    \]
    By~(\ref{endptsFixed}), $a > -k$, so 
    $y=(a-1,-k,n-k) \in S_k$, and our choice of~$a$
    implies
    \[ p_2(x) - p_2(y) > a - (a-1) = 1,
    \]
    which is impossible since $y \adj_{c_3} x$.
    \item Suppose $p_2(f(x)) < p_2(x)$. Without loss of
    generality, $x=(a',-k,n-k)$ where
    \[ a' = \max_{y \in S_k} \{p_2(y) \mid p_2(f(y)) \neq p_2(y) \}.
    \]
    By~(\ref{endptsFixed}), $a' < k$, so 
    $y=(a'+1,-k,n-k) \in S_k$, and our choice of~$a'$
    implies
    \[ p_2(x) - p_2(y) < a' - (a'+1) = -1,
    \]
    which is impossible since $y \adj_{c_3} x$.
 \end{itemize}
Thus we must have 
\begin{equation}
\label{p_2InQ_n}
    p_2(f(x)) = p_2(x).
\end{equation}

It is easily seen that there is a $c_3$-path~$P$ in~$Q_n$
from~$(a,-k,0) \in W_n \subset \Fix(f)$ through~$x$ 
to~$U \in \Fix(f)$ such that~$P$ is monotone increasing in
the $3^{rd}$ component. It follows from
Corollary~\ref{pullingSegment} that
\begin{equation}
\label{p_3InQ_n}
    p_3(f(x)) = p_3(x).
\end{equation}

From~(\ref{p_1InQ_n}), (\ref{p_2InQ_n}), and (\ref{p_3InQ_n}),
$x \in \Fix(f)$. Since~$x$ was taken arbitrarily,
$L \subset \Fix(f)$.

Similarly, $F \cup R \cup B \subset \Fix(f)$. Thus, so
far we have that $P_n \cup W_n \subset \Fix(f)$.

Now let $x = (a,b,h) \in Q_n \setminus (P_n \cup W_n)$.
There exist digital $c_3$-paths through~$x$,
\begin{itemize}
    \item $L_1$ from $(-n,b,h) \in B \subset \Fix(f)$ 
    to $(n,b,h) \in F \subset \Fix(f)$ that is monotone
    increasing in the first coordinate;
    \item $L_2$ from $(a,-n,h) \in L \subset \Fix(f)$
          to $((a,n,h) \in R \subset \Fix(f)$ that is monotone
    increasing in the $2^{nd}$ coordinate; and
    \item $L_3$ from $(a,b,0) \in W_n \subset \Fix(f)$ 
    to~$U \in \Fix(f)$ that is monotone increasing in the
    $3^{rd}$ coordinate.
\end{itemize}
By Corollary~\ref{pullingSegment}, $x \in \Fix(f)$.
We thus conclude that $f = \id_{Q_n}$, so~$A$ is a
freezing set.

That $A$ is the unique freezing set for $(Q_n,c_3)$ follows
from Lemma~\ref{W_nInFreezingSet}.
\end{proof}

\subsection{Cold 
sets for $P_n$ and $Q_n$}
Given a continuous $f: (X,\kappa) \to (X, \kappa)$,
a freezing set can so restrict~$f$ that~$f$ must 
be~$\id_X$. By contrast, a cold set, at most,
merely restriction on how much~$f$ may vary from an 
identity function; often, such restrictions are
satisfied by multiple functions.
Thus, while we often find that a digital
image has a unique minimal freezing set, the cold
and $(1,1)$-limiting sets shown below are not
shown to be minimal (indeed, in some cases they are
easily shown non-minimal).

As in section~\ref{FreezeP_nSection}, we use
an isomorphism between $(P_n,c_3)$ and
$([-n,n]_{\Z}^2, c_2)$
(Lemma~\ref{pyramidLikeSquare}).

\begin{thm}
    \label{ColdOrLimitedSquare}
    Let $A$ be any subset of
    \[ \{(-n,-n), (n,-n), (n,n), (-n,n)\}
     \subset [-n,n]_{\Z}^2.
    \]
    Then $Bd([-n,n]_{\Z}^2) \setminus A$ is
    a cold set 
    for $([-n,n]_{\Z}^2,c_2)$.
\end{thm}

\begin{proof}
If $A = \emptyset$ then, since a freezing set is
a cold set, the assertion follows from
Theorem~\ref{bdFreezes}. We proceed assuming
$A \neq \emptyset$.

Let $f \in C([-n,n]_{\Z}^2,c_2)$ such that 
\begin{equation}
\label{mostBdBfixed}
     \mbox{for }
x \in Bd([-n,n]_{\Z}^2) \setminus A,~~~
f(x) = x.
\end{equation}

We must show that for all~$x = (a,b) \in 
[-n,n]_{\Z}^2$,
$x \adjeq_{c_2} f(x)$. This is done as follows. 

If $x \in [-n,n]_{\Z}^2 \setminus Bd([-n,n]_{\Z}^2$, then $-n < a < n$ and $-n < b < n$.
Then $x \in \{(a,y) \mid -n \le y \le n \}$, a set
that is monotone in its second coordinate and 
whose endpoints, $(a,-n)$ and $(a,n)$, 
are in~$\Fix(f)$. It follows from 
Corollary~\ref{pullingSegment} that
$p_2(f(x)) = p_2(x)$. A similar argument shows
$p_1(f(x)) = p_1(x)$. Thus, $x \in \Fix(f)$.

Combining with~(\ref{mostBdBfixed}), we now have
$[-n,n]_{\Z}^2 \setminus A \subset \Fix(f)$.
Since $f \in C([-n,n]_{\Z}^2, c_2)$, we must have
$x \in A$ implies $f(x) \adjeq_{c_2} f(y)=y$ for
every $y \adj x$. 
Thus for all $x \in [-n,n]_{\Z}^2$, 
$x \adjeq_{c_2} f(x)$, as desired.
\end{proof}

\begin{remark}
It does not follow from the proof of
Theorem~\ref{ColdOrLimitedSquare} that~$f$
is an identity function. Indeed, since
$f \in C(X,c_2)$, we have the
following.
\begin{itemize}
    \item If $(-n,-n) \in A$ then 
    $f(-n,-n) \in \{ (-n,-n),~ (-(n-1), -(n-1)) \}$;
     \item if $(-n,n) \in A$ then 
           $f(-n,n) \in \{ (-n,n),~(-(n-1), n-1) \}$;
    \item if $(n,n) \in A$ then 
           $f(n,n) \in \{ (n,n),~(n-1, n-1) \}$;
    \item if $(n,-n) \in A$ then 
           $f(n,-n) \in \{ (n,-n),~(n-1, -(n-1)) \}$.
\end{itemize}
\end{remark}

\begin{cor}
    \label{P_nCold}
    Recall
    \[ T_n' = \{(-n,-n,0), (n,-n,0), (n,n,0), (-n,n,0)\}.
    \]
    Let $A'$ be any subset of $T_n'$.    
    Then $(T_n \setminus A')$ is 
    a cold set 
    for $(P_n,c_3)$.
\end{cor}

\begin{proof}
By Lemma~\ref{pyramidLikeSquare},
    there is an isomorphism 
    $F: ([-n,n]_{\Z}^2,c_2) \to (P_n,c_3)$
    with
    \begin{itemize}
        \item $F(Bd([-n,n]_{\Z}^2) = T_n$, and
        \item if $A$ is as in Theorem~\ref{ColdOrLimitedSquare}
    then $F(A) = A'$.
    \end{itemize}
    The asssertion follows
    from Theorem~\ref{ColdOrLimitedSquare} and
    Theorem~\ref{s-cold-invariant}.
\end{proof}

\begin{remark}
    Other cold sets for $(P_n,c_3)$ can
    be found by using similar methods along
    with Proposition~6.8
    of~{\rm \cite{BxFpSets2}}.
\end{remark}

\begin{remark}
    By an argument similar to that used to
    prove Corollary~\ref{P_nCold}, one can
    show the following.

    Let $A \subset W_n$ be any subset of~$T_n'$.
    Then $W_n \setminus A$ is a cold set for
    $(Q_n,c_3)$.    
\end{remark}

\section{Freezing and cold sets for suspension-like images}
Let
\[ H_n = P_n \cup P_n', \mbox{ where }
    P_n' = \{ \, (a,b,c) \in \Z^3 \mid
    (a,b,-c) \in P_n \, \}
\]
$H_n$ may be regarded as a digital
model of the topological suspension
of the square (border only) modeled
by $T_n$. $H_n$ may also be regarded as
made of 2 copies of $P_n$ with their
bases sewn together.

Let
\[ K_n = Q_n \cup Q_n', \mbox{ where }
    Q_n' = \{ \, (a,b,c) \in \Z^3 \mid
    (a,b,-c) \in Q_n \, \}
\]
$K_n$ may be regarded as a digital
model of the topological suspension
of the square (with interior) modeled
by $W_n$. $K_n$ may also be regarded as
made of 2 copies of $Q_n$ with their
bases sewn together.

As above, we let $U$ be
the highest point (and $L$, the lowest 
point) of $H_n$ (and of $K_n$):
\[ U = (0,0,n),~~~~~L = (0,0,-n).
\]
See Figure~\ref{fig:suspendedSquare}.

\begin{figure}
    \centering
    \includegraphics[height=3in]{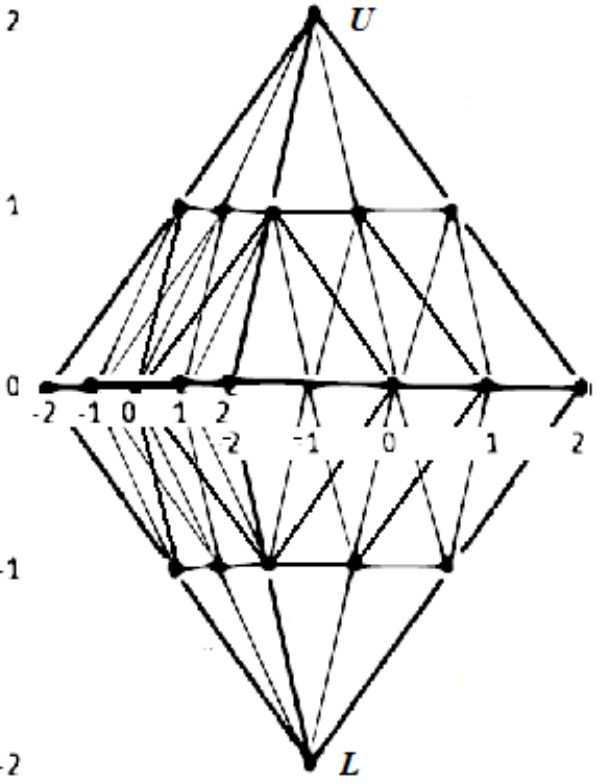}
    \caption{Left and front sides of
    the suspension-like
    digital image $(H_2,c_3)$}
    \label{fig:suspendedSquare}
\end{figure}

We have the following.

\begin{thm}
\label{H_nFreeze}
    $A = \{U,L\} \cup T_n$ is the unique minimal
    freezing set for~$(H_n,c_3)$.
\end{thm} 

\begin{proof}
Let $f \in C(H_n,c_3)$ such that
    $f|_{T_n} = \id_{T_n}$.

    First, we show $f(P_n) \subset P_n$.
    Let $x = (a,b,c) \in P_n$. Suppose,
    in order to obtain a contradiction, that $f(x) \not \in P_n$. 
    Then $p_3(f(x)) < 0$, so
    \[ d(U,x) = n-c < 
    n+\mid p_3(f(x)) \mid \, 
    = d(U,f(x)) = 
    d(f(U),f(x)),
    \]
    which contradicts the continuity
    of~$f$. Thus $f(P_n) \subset P_n$.

    Thus $f|_{P_n}: P_n \to P_n$
    satisfies $f|_{T_n} = 
    \id_{T_n}$. From
    Theorem~\ref{P_nMinFreeze} it
    follows that $f|_{P_n} = \id_{P_n}$.

    Similarly, $f|_{P_n'} = \id_{P_n'}$.
    It follows that $A$ is a freezing
    set for $(H_n,c_3)$.

Next, we claim the following.

\begin{equation} 
\label{polesAndT_nInK_nFreezingSet}
    A = \{U,L\} \cup T_n \mbox{ is a subset of 
    every freezing set for }(K_n,c_3).
\end{equation}

This follows from an argument similar to the
  proof of Proposition~\ref{W_nInFreezingSet}, where
  the additional point~$L=(0,0,-n)$ of~$A$ 
  corresponds to~$y$ of Proposition~\ref{nbrPropertyForEssentialFreezePt},
  where we take $y'=(0,0,-(n-1))$.

  The claim of uniqueness follows from
  (\ref{polesAndT_nInK_nFreezingSet}).
\end{proof}

\begin{thm}
\label{minK_nFreezingSet}
    $\{U,L\} \cup W_n$ is the unique minimal
    freezing set for $(K_n,c_3)$.
\end{thm}

\begin{proof}
    This follows from an argument similar to
    that used to prove Theorem~\ref{H_nFreeze}.
\end{proof}

\begin{remark}
    By using methods as in the proof of
    Theorem~\ref{ColdOrLimitedSquare}, we
    get the following:

    Let $A$ be any subset of $T_n'$.
    Then 
    \begin{itemize}
    \item $\{U,L\} \cup (T_n \setminus A)$ is a
    cold set for $(H_n,c_3)$.
    \item $\{U,L\} \cup (W_n \setminus A)$ is a
    cold set for $(K_n,c_3)$.
    \end{itemize}
\end{remark}

\section{Further remarks}
Models for digital topology of
the topological constructions
of cone and suspension were
developed by Lupton,
Oprea, and Scoville
in~\cite{LuptonEtal}.
We have obtained some
$(m,n)$-limiting sets for
these constructions. We have
also found freezing and cold sets for
some cone-like and suspension-like digital images.

We are grateful for the suggestions and
corrections of an anonymous referee.

\end{document}